\documentclass{article}

\usepackage{amsmath}
\usepackage{amssymb}
\usepackage{amscd}
\usepackage{graphicx}
\usepackage{psfrag}
\usepackage{bbm}

\newtheorem{theorem}{Theorem}[section]
\newtheorem{lemma}[theorem]{Lemma}

\newtheorem{proposition}[theorem]{Proposition}

\newtheorem{exa}[theorem]{Example}

\newtheorem{exas}[theorem]{Examples}

\newtheorem{defini}[theorem]{Definition}
\newenvironment{definition}{\begin{defini} \em}{\end{defini}}

\newtheorem{rema}[theorem]{Remark}

\newtheorem{remas}[theorem]{Remarks}

\newenvironment{proof}{{\noindent \sc Proof: } }{\mbox{ }\hfill$\Box$  
                        \vspace{1.5ex} \par}

\def\hfleche#1#2{\smash{\mathop{\vbox{\hbox to 8.5mm{{#1}}}}\limits^{#2}}}

\makeindex
                              
\begin{document}

\title {A note on 3-manifolds and  complex surface singularities\footnote{
Research partially supported by  CONACYT and PAPIIT-UNAM from M\'exico.  Accepted for publication in Math. Z. \newline $\quad$ 
{\it
Key-words:} Todd genus,  Milnor numbers, Gorenstein singularities, smoothings, cobordism, geometric structures, frames. 
 \newline   {\it Mathematics Subject Classification.} Primary   32S05,  32S10,  32S25, 32S55,  57M27;
  Secondary 14Bxx, 55S35, 57R20, 57R57, 57R77.  }
  }

\author{ Jos\'e Seade
   }

\date{}


\def\Z{{\Bbb Z}}
\def\R{{\Bbb R}}
\def\S{{\Bbb S}}

\def\D{{\Bbb D}}
\def\C{{\Bbb C}}
\def\Q{{\Bbb Q}}

\def\B{{\mathbb B}}
 \def\RP{\mathbb{R P}}
\def\F{\mathcal F}
\def\L{\mathcal L}
\def\a{\alpha}
\def\0{\underline 0}
\def\b{\beta }
\def\g{\gamma}
\def\ro{\varrho}
\def\na{\nabla}
\def\o{\omega}
\def\e{\varepsilon}
\def\d{\delta}
\def\SU{{\rm SU}(2)}
\def\O{\mathcal O}
\def\K{\mathcal K}
\def\t{\theta}
\def\vt{\vartheta}
\def\nn {\vskip 0.3cm \noindent }
\def\n {\noindent}
\def\nn {\vskip.2cm \noindent}
\def\l{\ell}
\def\la{\lambda}
\def\v {\vskip.1cm}
\def\vv {\vskip.2cm}
\def\grad{\overline\nabla_X f}
\def\s{\mathbb{S}}

\setcounter{section}{-0}


\pagestyle{headings}


\maketitle
\bibliographystyle{plain}

\begin{abstract}
This article is motivated by the original Casson invariant regarded as an integral lifting of the Rochlin invariant. We aim to defining an integral lifting of the Adams e-invariant of stably framed 3-manifolds, perhaps endowed with some additional  structure. We succeed in doing so for manifolds which are links of  normal complex Gorenstein smoothable singularities. These manifolds are naturally equipped with a canonical $\SU$-frame.  To start we notice that 
the set   of homotopy classes of $\SU$-frames on the stable tangent bundle of every closed oriented  3-manifold  is canonically a $\mathbb Z$-torsor. Then we define the $\widehat E$-invariant for the manifolds in question, an integer that modulo 24 is the Adams e-invariant. The $\widehat E$-invariant for the canonical frame  equals the Milnor number plus 1, so  this brings a new viewpoint on the Milnor number of the smoothable Gorenstein surface singularities. 
\end{abstract}

\section*{Introduction}
It is well-known that every closed oriented 3-manifold $M$ is parallelizable (see for instance \cite[p. 46]{Ki}), and each choice of a trivialization of its tangent  bundle determines specific Spin and ${\rm Spin}^c$ structures on $M$.
If we equip $M$  with a Spin structure $\mathcal S$,  one has its classical Rochlin invariant, which is defined as
$$\qquad {\mathcal R} (M, \mathcal S):=   \; \sigma(X)   \qquad \hbox{mod} \;  16 \;,$$
where $\sigma(X)$ is the signature of a compact Spin manifold $X$ which has $M$ as its Spin-boundary. The class of  this number     modulo 16 does not depend on the choice of the manifold $X$ due to Rochlin's signature theorem, stating that the signature of every closed Spin 4-manifold is a multiple of 16. 
This invariant has played for decades an important role in 3-manifolds theory. In the late 1980s, Andrew  Casson introduced an integer-valued invariant $\lambda(M)$ for homology spheres, that in some sense counts half the number of conjugacy classes of  irreducible $\SU$-representations of the fundamental group of $M$. An integral homology sphere has a unique Spin structure, up to isomorphism, and the Casson invariant  provides a lifting to $\Z$ of the Rochlin invariant. The literature on this topic is vast and we refer to \cite {AM} for a thorough account.  The celebrated Casson invariant conjecture for links of hpersurface singularities in $\C^3$ states that in this setting the invariant  essentially equals the signature of the corresponding Milnor fiber (see  \cite{Ne-Wa}).

We know from work by Hirzebruch and  Atiyah-Singer, that there are deep similarities between the signature and another important invariant, the Todd genus. If we now equip the 3-manifold $M$ with a trivialization $\mathcal F$ of its (stable) tangent bundle, one has its complex Adams e-invariant (see \cite{CF}):
$$ \qquad  \qquad  e_c(M, {\mathcal F}) := {\rm Td}(X, {\mathcal F})[X]  \qquad  \hbox{mod} \; \Z \;,
$$
where $X$ is now a compact weakly complex manifold with boundary $M$; ${\rm Td}(X, {\mathcal F})[X]$ denotes the 2nd Todd polynomial (see \cite{Hir} or Section \ref{Section-preliminaries} below) evaluated in the Chern classes $c_i(X, {\mathcal F})$ of $X$ relative to the frame ${\mathcal F}$ on the boundary; $[X]  \in H^4(X,M)$ is the orientation cycle.  Thus we may equivalently define:
$$ e_c(X, {\mathcal F}) = \big(c_1(X, {\mathcal F})^2 + c_2(X, {\mathcal F}) \big)[X]  \quad \hbox{mod} \; 12 \,,
$$
so one has an invariant of the framed manifold $(X, {\mathcal F})$ in $\Z_{12}$. Bearing in mind the similarities between the signature and the Todd genus, and that for oriented integral homology spheres,  the Casson invariant is an integral lifting of the Rochlin invariant, it is natural to ask whether there  exists an integral lifting of the complex Adams $e_c$-invariant of  framed 3-manifolds, perhaps equipped with some additional structure? 

Here we  answer positively this question for  links of  normal Gorenstein smoothable surface singularities. This leads to an invariant   of surface singularities that we denote by 
$\widehat E(V,\0) \in \Z$. We prove:

\vv

\vv
{\bf Theorem 1} {\it  Let $L_V$ be the link of a normal surface singularity germ $(V,\0)$. Set $V^* := V \setminus \{\0\}$ and $ \tau L_V := TV^* |_{L_V}$ its stable tangent bundle. Then:

\begin{enumerate}
\item The set $\{{\mathcal Fr}_{\SU}(L_V)\}$ of homotopy classes of $\SU$-frames on  $\tau L_V$ is canonically a $\mathbb Z$-torsor and so it corresponds bijectively  with the integers. 

\item When  the germ $(V,\0)$ is Gorenstein, one has a canonical element $[L_V,\rho]$ in this set, depending only on the analytic type of the germ $(V,\0)$. 

\item If the germ $(V,\0)$ further is smoothable, then the bijection $\{{\mathcal Fr}_{\SU}(L_V)\} \leftrightarrow \Z$ becomes canonical and we get a canonical invariant $\widehat E(L_V,\rho) \in \Z$.

\item In this setting, with $(V,\0)$ Gorenstein and smoothable, we have 
$$ \widehat E(L_V,\rho) \, = \, 12 \,{\rm Td}(F_t, {\mathcal F})[X] = \mu_{GS} + 1 \;,$$
where $F_t$ is the Milnor fibre of a (any) smoothing and $\mu_{GS}$ is the Milnor number of the singularity, introduced by Greuel and Steenbrink, which equals the 2nd Betti number of the Milnor fibre.
\end{enumerate}
}

\vv
\vv
The first statement above actually holds for all closed oriented 3-manifolds (Lemma \ref {SU-frames}). Theorem 1 yields to a definition of the invariant 
$ \widehat E(L_V,\mathcal F) $ for all $\SU$-frames on the link of Gorenstein smoothable singularities, thus providing an integral lift of the complex Adams e-invariant.
There is  a natural projection map $\{{\mathcal Fr}_{\SU}(L_V)\} \longrightarrow \Omega_3^{fr}$ onto the stably framed cobordism group in dimension 3.  There is too a  real Adams e-invariant  $\Omega_3^{fr} \buildrel {e_R} \over {\longrightarrow} \Z_{24}$ which  is an ismorphism. One has 
$e_c = \frac{1}{2} e_R$ mod $\Z$, so  ${e_R}$ is slightly finer that ${e_c}$ (see \cite {Se4}), and  we actually get the commutative diagram below: 
 
\[
\begin{CD}
\Z \cong \{{\mathcal Fr}_{\SU}(L_V)\}   \buildrel {\widehat E} \over {\longrightarrow} \; \Z \\
\qquad  \quad \qquad \downarrow  \qquad  \quad \qquad    \downarrow \\
\qquad \qquad  \qquad \quad \quad   \quad  \Omega_3^{fr}  \quad \buildrel {e_R} \over {\longrightarrow} \quad \Z_{24} \qquad \quad \\ 
\qquad  \qquad \qquad  \quad \cong \; \downarrow  \qquad  \quad \qquad    \downarrow  \, \hbox{2-to-1}\\
\qquad \qquad  \qquad \quad \quad \quad  \Omega_3^{fr}  \quad \buildrel {e_c} \over {\longrightarrow} \quad \Z_{12} \qquad \quad
\end{CD}
\]

It would be interesting to extend these ideas and results to all framed 3-manifolds, perhaps equipped with some additional structure. In the case we envisage here one has canonical contact and ${\rm Spin}^c$ structures on the link, compatible with the frame  (see for instance \cite{CNP, NeSzi, NeSigu}); these might be clues for a generalization. 

The class of singularities to which Theorem 1 above applies  includes  all hypersurface and  ICIS germs. The  3-manifolds which arise as links of surface singularities are the oriented graph-manifolds with negative definite intersection matrix (see  \cite {Ne1} for a thorough account on the subject,  and \cite[Chapter 4]{Seade-librosing} for an  introduction). 
One has that given an arbitrary 
 finite graph $\Gamma$, say  with  $r$ vertices, for almost all sets of negative weights $w= (w_1,...,w_r)$ for the vertices the corresponding intersection matrix is negative definite; and if the graph has no loops, then for each such vector of weights, there are infinitely many vectors of genera $g= (g_1,...,g_r)$, $g_i \ge 0$, such that the  weighted graph $(\Gamma, w, g)$ corresponds to some Gorenstein singularity (although not every  such singularity is smoothable). This follows easily from \cite{LS, PS, PPP}.


\section{Preliminaries}\label{Section-preliminaries}
The material in this section is all classical and essentially comes from \cite{KM, Hir, CF}; we include it here for completeness. By the {\it stable tangent bundle} $\tau M$ of an oriented smooth $n$-manifold $M$ we mean the direct sum $\tau M:= (1) \oplus TM$ of a trivial line bundle and its tangent bundle. 

If $\tau M$ is trivial, hence isomorphic to the bundle $M \times \R^{n+1}$, 
by a (stable) {\it frame} on $M$ we mean a specific choice of a trivialization of $\tau M$. Hence we may regard a frame on $M$ as a set  $\alpha = (\alpha_1,\cdots,\alpha_{n+1})$ of $n+1$ linearly independent sections of $\tau M$. 

If $A: M \to {\rm GL}(n+1,\R)$ is a smooth function and $\alpha$ is as above, we can twist this frame using $A$ in the obvious way: At each $x \in M$  the basis $(\alpha_1(x),\cdots,\alpha_{n+1}(x))$  of the vector space $\tau_xM$ is carried by the matrix $A(x)$ into a new basis of $\tau_xM$. Conversely, given two stable frames $\alpha = (\alpha_1,\cdots,\alpha_{n+1})$ and $\beta = (\beta_1,\cdots,\beta_{n+1})$ on $M$, at each point $x$ we have two possibly different bases of $\tau_xM$, and we can pass from one to the other by a linear isomorphism. So the two frames differ by a smooth function $d(\alpha,\beta): M \to {\rm GL}(n+1,\R)$, and the homotopy classes of frames on $M$ are classified by $[M,{\rm GL}(n+1,\R)]$, the homotopy classes of maps from $M$ into the General Linear group in dimension $n+1$.

If $n$ is odd, say $n=2k- 1$, and $\tau M$ is equipped with the  structure of a complex vector bundle, we may speak of {\it complex frames}, which means a set of $k$ sections of $\tau M$ which are linearly independent over $\C$. In this case the
complex frames on $M$ are classified by the set $[M,{\rm GL}(k,\C)]$.

Notice that similar considerations apply to vector bundles in general. In the sequel we will actually consider a non-singular complex analytic surface $V^*$, its canonical bundle $\mathcal K_{V^*}$, which is a holomorphic line bundle, and nowhere-vanishing holomorphic sections of it, {\it i.e.}, holomorphic 2-forms on $V^*$ which do not vanish at any point. Every such 2-form can be regarded as a holomorphic 1-frame  on $\mathcal K_{V^*}$. In this case the difference between two such 2-forms is just the quotient of the two sections, which is a never-vanishing holomorphic function with values in $\C^*$, the non-zero complex numbers.

Now recall \cite{Hir}. The {\it Todd sequence of polynomials} is the multiplicative sequence $\{T_k(c_1,\cdots,c_k)\}$ with
characteristic power series:
$$Q(x) = \frac{x}{1 - e^{-x}} =  1 + \frac{1}{2} x + \sum_{k=1}^\infty (-1)^{k-1} \frac {{\mathcal B}_k}{(2k)!} x^{2k} \;,
$$
where $\{{\mathcal B}_k\}$  are the Bernoulli numbers. The first polynomials in this sequence are:
$$ T_1 = \frac {1}{2} c_1 \quad ; \;  T_2 = \frac {1}{12} (c_2 + c_1^2) \quad ; \;  T_3 = \frac {1}{24} (c_2  c_1)  \quad ; \; T_4 = \frac {1}{720} (-c_4 + c_3 c_1 + 3 c_2^2 + 4c_2 c_1^2 -c_1^4)  \quad ; \; \cdots 
$$
If $X$ is now an almost complex closed manifold of real dimension $2k$, then its {\it Todd genus} is defined as:
$${\rm Td}[X] = T_k(c_1(X),\cdots,c_k(X))[X]\,,$$
where the $c_i(X)$ are the Chern classes and $[X]$ is the orientation cycle. This is in principle a rational number, but the Hirzebruch-Riemann-Roch theorem  in \cite{Hir} says that for complex manifolds it equals the analytic Euler characteristic, so it is an integer. This statement was generalized by the Atiyah-Singer index theorem to arbitrary almost complex manifolds using the $\bar \partial$-operator.

Now suppose $X$ is an almost complex compact manifold of real dimension $2k$ with non-empty boundary $M$. Then the restriction of the tangent bundle $TX$ to $M$ can be regarded as the stable tangent bundle $\tau M = TX|_M = (1) \oplus TM\,$, where $(1)$ is the normal bundle of $M$ in $X$, which is a trivial line bundle.  If we have a complex frame $\alpha = (\alpha_1,\cdots,\alpha_{k})$ on $M$, this defines representatives of the Chern classes of $X$ that vanish over $M$. These are by definition the {\it Chern classes of $X$ relative to the frame  on the boundary}, $c_i(X;\alpha)$, $i=1,\cdots,k$; these live in $H^*(X,M;\mathbb Z)$. In this setting one has  the corresponding Todd genus of $X$ relative to the frame  $\alpha$, see \cite{CF}:
$${\rm Td}[X;\alpha] = T_k(c_1(X;\alpha),\cdots,c_n(X;\alpha))[X,M]\,.$$
It is clear from the definition of the relative Chern classes that these map to the usual Chern classes in $H^*(X)$ under the morphism induced by the inclusion $(X,\emptyset) \to (X,M)$. It follows, using basic properties of the cup product, that if we have two different complex frames on the boundary of $X$, then all the decomposable relative Chern numbers coincide and the difference in their corresponding relative Todd genus is determined only by $c_n$, the class of top degree. In the sequel we consider only the case $k=2$.

\section{Proof of the theorem}

Let $M$ be a closed oriented 3-manifold and $TM$  its tangent bundle,  isomorphic to $M \times \R^3$ since every such manifold has trivial tangent bundle.
Let $\tau M := (1)  \oplus TM $ be its stable tangent bundle, where $(1)$ is a trivial real line bundle over $M$.  Then $\tau M$ can be identified with $M \times \C^2$. By an $\SU$-frame  on $M$ we mean  a basis $\beta = (\beta_1,\beta_2)$  for $\tau M$ so that at each $x \in M$ the  two vectors $\beta_i(x)$ define a matrix in $\SU$. Notice that any two $\SU$-frames on $M$  differ  by a map from $M$  into the  Lie group $\SU$, which is isomorphic to the 3-sphere $\s^3$. The homotopy classes of maps from $M$ into  $\s^3$, $[M,\s^3]$,   are classified by their degree.
Hence, letting $\{{\mathcal Fr}_{\SU}(M)\}$ denote the set of homotopy classes of $\SU$-frames on $M$, we get:

\begin{lemma}\label{SU-frames}
There is a bijection between the sets $\{{\mathcal Fr}_{\SU}(M)\}$ and $[M,\s^3] \cong \Z$. In fact $\{{\mathcal Fr}_{\SU}(M)\}$ is 
canonically a $\mathbb Z$-torsor.
\end{lemma}

Recall that a $\mathbb Z$-torsor is a principal homogeneous space for the group $\mathbb Z$ of the integers, {\it i.e.},  a non-empty set on which $\mathbb Z$  acts freely and transitively.
Notice that the bijection $\{{\mathcal Fr}_{\SU}(M)\} \cong \Z$
is not canonical in general.
Now let $(V,\0)$ be a normal complex surface singularity in some $\C^N$, so $V$ is a complex 2-dimensional variety,   non-singular away from $\0$ and every $\C$-valued holomorphic function defined on a punctured neighbourhood of $\0$ in $V$ extends to $\0$.
We know from  \cite{Mi1} that for  $\e > 0$ sufficiently small, every sphere $\s_r$  in $\C^N$ centered at $\0$ and of radius $\le \e$ meets $V$ transversally and the pair $(\B_\e, V \cap \B_\e)$ is homeomorphic to the cone over $(\S_\e, V \cap \s_\e)$, where $
\B_\e$ is the ball bounded by the sphere $\s_\e$.  Such an $\s_\e$ is called a Milnor sphere for $V$. The 3-manifold $L_V := V \cap \s_\e$ is called the link of the singularity and its oriented diffeomorphism type is independent of the choice of $\e$ provided this is small enough.

  Let us  equip $V^*:= V \setminus \{\0\}$
 with the induced Hermitian metric from $\C^N$, so the structure group of its tangent bundle $TV^*$ can be assumed to be ${\rm U}(2)$. Let the germ $(V,\0)$  be Gorenstein and
 let $\omega$ be a nowhere vanishing holomorphic 2-form on $V^*$. We know from \cite{Se1} that $\omega$ defines a reduction to $\SU$ of the structure group of the  bundle $TV^*$. Since $\SU \cong \s^3$ we may think of it as being the symplectic group  $Sp(1)$ of unit quaternions. 
  The $\SU$-structure on $V^*$ defines multiplication by the quaternions $i, j, k$ at the tangent space of each point in $V^*$. Now let $\nu$ denote the unit outwards normal field of $M$ in $V^*$. Multiplying this normal vector field by the quaternions $i, j, k$ at each $x \in M$ we get an $\SU$-frame  on $M$. This is the {\it canonical framing} $\rho$ from \cite{Se1}. 
  We have:
  
  \begin{lemma}
   The  frame  $\rho$  is independent of the choice of holomorphic 2-form up to homotopy.
  \end{lemma}

 \begin{proof} 
Let $\omega_i$, $i=1,2$, be two never-vanishing holomorphic 2-forms on $V^*$. Their quotient is a never-vanishing holomorphic function $g: V^* \to \C$. Since the germ $(V,\0)$ is assumed to be normal, this extends to a holomorphic map $\tilde g: V \to \C$ that does not vanish  at $\0$, because a holomorphic function on $V$ cannot vanish at a single point. Hence $\tilde g$ actually takes values in $\C^*$. Since $V$ is homeomorphic to the cone over $M$ and a continuous function between topological spaces is nulhomotopic if and only if it extends to the cone, it follows that  the restriction of $g$ to the link $M$ is homotopic to a non-zero constant function. This proves the lemma.
\end{proof}

So we now have a canonical way  for  associating  to each normal Gorenstein surface singularity $(V,\0)$ a homotopy class of 
  $\SU$-frames on its link and we arrive to the following: 
    
\begin{theorem}\label{mod24}
Let $L_V$ be the link of a normal Gorenstein surface singularity $(V,\0)$. Then a choice of a nowhere vanishing holomorphic 2-form on $V^*:= V \setminus \{\0\}$ determines a canonical $\SU$-frame  $\rho$ on $L_V$, which is independent of the choice of the 2-form up to homotopy. Then $(L_V,\rho)$ represents a canonical  element in $\{{\mathcal Fr}_{\SU}(L_V)\}$, the set of homotopy classes of $\SU$-frames on $L_V$, which is classified by the integers. 
\end{theorem}

We remark that for every closed oriented 3-manifold $M$ there is a natural surjective projection from the set 
${{\mathcal Fr}_{\SU}(M)}$ into 
the framed cobordism group $\Omega_3^{fr} \cong \Z_{24}$.
The image in  $\Omega_3^{fr}$ of the canonical element $(L_V,\rho) \in \{{\mathcal Fr}_{\SU}(L_V)\}$ can be determined by the real Adams e-invariant  $e_R(L_V, \rho)$. We know from \cite{Se4}  that this invariant can be expressed as
$$ e_R(L_V, \rho) =  12 \big({\rm Td}(X,\rho)[X] +   \hbox{Arf} \, (K_X) \big)  \quad \hbox{mod} \; (24) \,,$$
where $X$ is a compact weakly complex manifold with boundary $L_V$, ${\rm Td}(X,\rho)[X] $ is as before, the Todd genus of $X$ in the Chern classes of $X$ relative  to $\rho$, and ${\rm Arf} (K_X)  \in \{0,1\} $ is the Arf invariant of a certain quadratic form associated to the 1st Chern class $c_1(X;\rho)$ of $X$ relative to the framing. If $X$ further has an $\SU$-structure compatible with the framing, then $c_1(X;\rho) = 0$ and one gets (cf. \cite{CF}):

\begin{equation}\label{equation1.3} \; \; e_R(L_V, \rho) =  12 c_2(X;\rho) [X]  \quad \hbox{mod} \; \; (24) \end{equation}
$$ \qquad \qquad  \quad \; \;= 12 \chi(X)  \qquad  \qquad \hbox{mod} \; (24) \,.$$

 Recall now that the germ 
 $(V,\0)$ is 
smoothable if there exists a 
3-dimensional normal complex analytic space $\mathcal W$ with an isolated singular point also denoted $\0$, and a 
flat morphism:
$$\mathcal G: \mathcal W \longrightarrow \C \,,$$
such that ${\mathcal G}^{-1}(0)$ is $V$ and  ${\mathcal G}^{-1}(t)$ is non-singular for all $t \ne 0$ with $|t|$ sufficiently small. In this case one has a similar picture to that of the Milnor fibration of a hypersurface singularity, the difference being that the ambient space may now be singular.

Let  $\mathcal G: \mathcal W \longrightarrow \C \,$ be a smoothing of the germ $(V,\0)$ and assume $ \mathcal W $ is defined in $\C^N$. There exists $\e> 0$ sufficiently small such that the sphere $\s_\e$  centered at $\0 \in \mathcal  W$  meets transversally $V$ and $\mathcal  W$ and  is a Milnor sphere for both spaces. We  set $F:= {\mathcal G}^{-1}(t) \cap \B_\e$ for some $t \ne 0$ with $|t|$ sufficiently small, where $\B_\e$ is the ball in $\C^N$ bounded by $\s_\e$.  We call $F$ the Milnor fibre of the smoothing.
 We know from \cite{Gr-Sten} that if the germ is Gorenstein and smoothable, then the Euler characteristic $\chi(F)$ 
 is independent of the choice of smoothing at it equals $\mu_{GS} +1$ where $\mu_{GS}$ is the Milnor number of the singularity. By definition $\mu_{GS}$ is the 2nd Betti number of  the corresponding Milnor fibre.

The 
 following lemma from \cite{Se2, Se3} is essential for what follows.

\begin{lemma} \label{Durfee conjecture} Let $\mathcal G: \mathcal W \longrightarrow \C \,$ be a smoothing of the Gorenstein germ $(V,\0)$.
If $\omega$ is a nowhere vanishing holomorphic 2-form on $V \setminus \{\0\}$, then $\omega$ and the flat map $\mathcal G$ determine a nowhere vanishing holomorphic 3-form on  $\mathcal W\setminus \{\0\}$, and this determines (by contraction) a nowhere vanishing holomorphic 2-form on each Milnor fibre $F$ of the smoothing.
\end{lemma} 

This lemma implies the following theorem which essentially comes  from \cite{Se2}. A slightly weaker  version of the second statement was conjectured by Durfee \cite[Conjecture (1.6)]{Du}.

\begin{theorem}\label{Durfee thm} 
Let $F$ be a Milnor fibre  of a  smoothable normal Gorenstein complex surface singularity. Then:
\begin{enumerate}
\item The tangent bundle $TF$ has an $\SU$-structure extending the canonical one on the boundary.
\item 
The bundle $TF$ is 
topologically trivial as a complex bundle. 
\end{enumerate}
\end{theorem}

The first statement in  \ref{Durfee thm} is immediate from \ref{Durfee conjecture} and the previous discussion. The second statement follows from the first one and the fact that there exist on $F$ nowhere vanishing vector fields because $H^4(F; \mathbb Z) =0$ since $F$ is a connected manifold with non-empty boundary. Multiplying one such vector field by the quaternions $i, j, k$ one gets a trivialization of $TF$ compatible with its complex structure.

From equation  \ref{equation1.3}  and Theorem \ref {Durfee thm} 
we deduce: 
$$ e_R(L_V, \rho) \; =  \,  12 \chi(F)  \quad \hbox{mod} \; (24) \; $$
$$\qquad \qquad \quad  = \; \mu_{GS} +1 \quad \hbox{mod} \; (24) \,,
$$
where $\mu_{GS} $ is the Milnor number of $(V,\0)$ introduced in \cite{Gr-Sten},  independent of the choice of the smoothing.

To complete the proof of Theorem 1, and get the commutative diagram stated in the introduction, we still need to show that there is a natural bijection between the set $\{{\mathcal Fr}_{\SU}(M)\}$ and the additive group $\Z$, taking the element $(L_V,\rho)$ into the stated integer. It is here that we use the condition of smoothability.

We  now want to define a natural bijection $\{{\mathcal Fr}_{\SU}(M)\} \, \leftrightarrow \, \Z$ giving rise to
 the $\widehat E$-invariant associated to the germ $(V,\0)$.    We do it as follows. First, let $\beta$ be a trivialization of the bundle $TF$ defined by a  global $\SU$-frame.
This gives a ``base point" in $\{{\mathcal Fr}_{\SU}(M)\}$,  the set of homotopy classes of $\SU$-frames on $M$. 
Proposition \ref{prop-e-inv} below  shows that  this base point is canonical.  

Once we have the base point determined as above, we  define a bijection
$\{{\mathcal Fr}_{\SU}(M)\} \buildrel {\delta} \over {\longrightarrow} \Z
$
using the $\mathbb Z$-action: Set $\delta(M,\beta) = 0 $,  and  given an arbitrary element $(M,\mathcal F) \in 
\{{\mathcal Fr}_{\SU}(M)\}$ notice that  $\mathcal F$ differs from $\beta$ by a map $d(\mathcal F, \beta): M \to \SU$. By Hopf's degree theorem (see for instance \cite [Chapter 3]{Gu-Po}), the map $d(\mathcal F, \beta)$   has a certain degree $d \in \Z$ that characterizes it up to homotopy. Set $\delta(M,\mathcal F) = d$. 
Then
define:

\begin{definition}
The $\widehat E$-invariant of the germ $(V,\0)$  is   the difference (measured as a degree) between the canonical frame $\rho$ on the link $M$, and the restriction to $M$ of a trivialization of the tangent bundle $TF$:
$$\widehat E (V,\0) := \delta (M,\rho) \in \Z \,.$$
\end{definition}

The proposition below shows that this invariant is independent of all choices:

\begin{proposition}\label{prop-e-inv} 
One has $\widehat E (V,\0) = 12{\rm Td}(F,\rho) [F] = \mu_{GS} + 1 $, where $\mu_{GS} $ is the Milnor number.
\end{proposition}

\begin{proof} Since $TF$ is a trivial bundle, its first Chern class  vanishes, $c_1(F) =0$. By a general property of the cup product, given an arbitrary  $\SU$-frame  on the link, for the corresponding relative  first Chern class   $c_1(F; \mathcal F)$ one has 
$$ c_1(F; \mathcal F)^2 = c_1(F; \mathcal F) \cdot c_1(F) = 0 \,.$$
Hence $ {\rm Td}(F,\rho) [F] = \frac{1}{12} \chi(F)$.
We know from \cite {Gr-Sten} that the Euler characteristic of $F$ is $\mu_{GS}  +1$. Choose small closed balls $D_1, \cdots, D_{\mu +1}$ in the interior of $F$, pairwise disjoint. Let $F^*$ denote the compact manifold obtained by removing from $F$ the interior of those balls. Then $\chi(F^*) = 0$. 
Now choose at the boundary of each $D_i$ a vector field $\nu_i$ pointing toward the center of the ball. 
Put on $M = \partial F$ the unit outwards normal vector field. Since $\chi(F^*) = 0$, these vector fields on $\partial F^*$ extend to a vector field $\zeta$ on $F^*$ with no singularity in its interior, by the theorem of Poincar\'e-Hopf for manifolds with boundary (see for instance \cite{Mi2}.
The $\SU$-structure determines multiplication by the quaternions $i, j, k$ defined at each point in $F^*$. Doing so we get on $M$ its canonical framing, and we get on the boundary of each $D_i$ the canonical frame on the 3-sphere. This determines an extension of the map $d(\mathcal F, \beta): M \to \SU$ to a map $\tilde d(\mathcal F, \beta): F^* \to \SU$. Hence the degree of $d(\mathcal F, \beta)$ is the sum of the degrees on the spheres. Since each of these contributes with 1,  the result follows from Lemma \ref{Durfee conjecture}.
\end{proof}

\begin{remas} $\,$

\begin{enumerate} {\rm
\item The $\widehat E$ invariant defined above can be computed using the Riemann-Roch defect introduced by 
E. Loojienga  in \cite [Section 3]{Lo}. 

\item Given an $\SU$-frame  on a closed oriented 3-manifold $M$ which bounds a compact weakly complex 4-manifold $X$, the obstruction to extending the underlying $\SU$-structure to the interior of $X$ yields to an obstruction in dimension 2, that can be represented by an oriented 2-submanifold $C$, which is the Lefschetz dual of the 1st Chern class of $X$ relative to the boundary. The manifold $C$ 
is a characteristic submanifold of $X$ in the sense of 
\cite {Fr-Kir, Ki}. By \cite [Chapter XI]{Ki}, the characteristic cobordism group  in dimension 4 is isomorphic to $\Z \oplus \Z$, the isomorphism being achieved by taking a characteristic pair $(X,C)$ to the pair of integers $\big(\sigma(X), \frac{1}{8} (\sigma(X) - C^2) \big)$. In the setting we envisage here, the second invariant essentially is the Todd genus. This suggests that the results in this paper can be regarded from the viewpoint of characteristic cobordism.

\item Whenever one has a characteristic pair $(X,C)$ as above, $X$ becomes a ${\rm Spin}^c$ manifold and $C$ is Spin. These structures play a significant role in low dimensional manifolds and in the theory of surface singularities, see for instance \cite {NeSzi, NeSigu}. I believe that the $\widehat E$-invariant, and Looijenga's Riemann-Roch defect, must have a deep relation with the Seiberg-Witten invariant of the link with its canonical  ${\rm Spin}^c$-structure.
}
\end{enumerate} 
\end{remas}

\vglue.2in
\begin{tabular}{ll}
Jos\'e Seade\\
 Instituto de Matem\'aticas,    \\
 Universidad Nacional Aut\'onoma de M\'exico.\\
  jseade@im.unam.mx
\end{tabular}


{\small
\begin{thebibliography}{000}

\bibitem{AM}
S. Akbulut,  J. McCarthy. {\em  Casson's invariant for oriented homology 3-spheres - an exposition}. Mathematical Notes 36 (1990). Princeton University Press, Princeton, NJ. 

\bibitem{CNP}
C. Caubel, 
A. N\'emethi, P. Popescu-Pampu. {\em 
Milnor open books and Milnor fillable contact 3-manifolds}. 
Topology 45 (2006), 673--689.


\bibitem{CF} P. E. Conner, E. E. Floyd. {\em The relation of cobordism to K-theories}. Lecture Notes in Maths. 28 (1966).

\bibitem{Du} 
A. Durfee. {\em The signature of smoothings of complex surface singularities}. Math.
Ann.  232  (1978), 85--98.


\bibitem{Fr-Kir}
M. Freedman, R.  C. Kirby. {\em
A geometric proof of Rochlin's theorem}. In
``Algebraic and Geometric Topology", Stanford/Calif. 1976. Proc. Symp. Pure Math. A.M.S.  32, Part 2 (1978), 85--97.

\bibitem{Gr-Sten}
G.-M. Greuel, J. Steenbrink. {\em On the topology of smoothable singularities}, Proc.
Symp. Pure Math. A.M.S. 40, Part 1 (1983),  535--545.

\bibitem{Gu-Po} V. Guillemin, A. Pollack. {\em Differential Topology}.  Reprint of the 1974 original.
Providence, RI: AMS Chelsea Publishing.

\bibitem{Hir} 
F. Hirzebruch. {\em  Topological Methods in algebraic geometry}. Springer Verlag, 1956.


\bibitem{Ki}
R.  C. Kirby. {\em The topology of 4-manifolds}. Lecture Notes in Maths. 1374 (1989). Springer Verlag.

\bibitem{KM} M. Kervaire, J. Milnor. {\em Groups of Homotopy spheres. I.}  Ann. Math. 77 (1963), 504--537.


\bibitem{LS}
F. Larri\'on, J. Seade. {\em
Complex surface singularities from the combinatorial point of view}. 
Topology Appl. 66 (1995), 251--265.

\bibitem{Lo} E. Looijenga. {\em Riemann-Roch and smoothings of singularities}. Topology 25 (1986), 293--302.


\bibitem{Mi1}
J. Milnor. {\em Singular points of complex hypersurfaces}. Annals
of Mathematics Studies {61}, Princeton University Press,
  Princeton, N.J., 1968.

\bibitem{Mi2} 
J. Milnor. {\em Topology from the Differentiable Viewpoint}. Univ. Press of Virginia, Charlottesville, 1965.

\bibitem{NeSzi} A. N\'emethi, A. Szil\'ard.  {\em Milnor Fiber Boundary of a Non-isolated Surface Singularity}. Lecture Notes in Maths. 2037,  Springer Verlag, 2012.

\bibitem{NeSigu} A. N\'emethi, B. Sigurdsson.
{\em The geometric genus of hypersurface singularities}.  J. Eur. Math. Soc.  18 (2016), 825--851.

\bibitem{Ne1} W. D. Neumann. {\em A calculus for plumbing applied to the topology of complex surface singularities and degenerating complex curves}. Trans. Amer. Math. Soc., 268 (1981), 299--344.

\bibitem {PS} P. Popescu-Pampu, J. Seade. {\em A finitness theorem for dual graphs of surface singularities}. Int. J. Maths.20 (2009), 1057--1068.

\bibitem {PPP} P. Popescu-Pampu. {\em Numerically Gorenstein surface singularities are homeomorphic to Gorenstein ones}. Duke Math. J. 159 (2011), 539--559.

\bibitem{Ne-Wa} W. D. Neumann, J. Wahl. {\em 
Casson invariant of links of singularities}. Comm. Math. Helv. 65 (1990), 58--78.

\bibitem{Roch1}
V. A. Rochlin. {\em New results in the theory of four-dimensional manifolds}, Doklady Acad. Nauk. SSSR (N.S.) 84 (1952), 221--224. 

\bibitem{Se1} 
J. Seade. {\em Singular point of complex surfaces and homotopy}. Topology 21 (1982), 1--8.

\bibitem{Se2}  J. Seade. {\em  A cobordism invariant for surface singularities}. Proc. Symp. Pure Math. 40, part 2 (1983), 479--484.


\bibitem{Se3}  J. Seade. {\em Vector fields on smoothings of complex singularities}. In
Topics in several complex variables, (ed) E. Ram\'irez de Arellano, D. Sundararaman. 
Res. Notes Math. 112 (1985), 152--157, Pitman Advanced Publishing Program. 


\bibitem{Se4}  J. Seade. {\em 
A note on the Adams e-invariant}. In ``Homotopy theory and its applications. A conference on algebraic topology in honor of Samuel Gitler", August 1993, Cocoyoc, Mexico. Ed. A.
Adem  et al. A. M. S.  Contemp. Math. 188 (1995), 223--229.


\bibitem{Seade-librosing} 
J. Seade. {\em  On the topology of isolated singularities in analytic spaces}. Progress in Mathematics 241. Basel: Birkh\"auser, 2006.


\end{thebibliography}
}
\end{document}